\documentclass{article}

\usepackage{graphicx}
\usepackage{amsmath}
\usepackage{amsthm}
\usepackage{amstext}
\usepackage{amsfonts}
\usepackage{amssymb}
\usepackage{amsxtra}
\usepackage{bbm}
\usepackage{url}
\usepackage{subfigure}
\usepackage{tikz}
\usepackage{booktabs}
\usepackage{pgfplots}
\pgfplotsset{compat=newest}
\pgfplotsset{plot coordinates/math parser=false}
\newlength\figureheight
\newlength\figurewidth 
\usepackage[square,numbers,sort&compress]{natbib}
\usepackage{algorithm}
\usepackage{algorithmic}

\newtheorem{thm}{Theorem}[section]

\newtheorem{lem}[thm]{Lemma}

\theoremstyle{definition}

\newtheorem*{cexmp*}{Counterexample}
\newtheorem{rem}[thm]{Remark}

\newcommand \be {\begin{equation}}
\newcommand \ee {\end{equation}}
\newcommand \ben {\begin{equation*}}
\newcommand \een {\end{equation*}}
\newcommand \bea {\begin{eqnarray}}
\newcommand \eea {\end{eqnarray}}
\newcommand \bean {\begin{eqnarray*}}
\newcommand \eean {\end{eqnarray*}}

\newcommand \RR {\mathbb{R}}

\newcommand \EE {\mathbb{E}}

 \DeclareMathOperator{\diag}{Diag}

\DeclareMathOperator{\conv}{conv} 
\DeclareMathOperator{\rg}{rg}
\DeclareMathOperator{\proj}{proj}
\DeclareMathOperator{\trace}{trace}

\newcommand{\scp}[2]{\langle #1 \,,\, #2\rangle}
\newcommand{\norm}[2][]{\|#2\|_{#1}}

\begin{document}
\title{The Randomized Kaczmarz Method with Mismatched Adjoint}

\author{Dirk~A. Lorenz\thanks{Institute for Analysis and Algebra, TU Braunschweig, 38092 Braunschweig, Germany, \texttt{d.lorenz@tu-braunschweig.de}} \and Sean~Rose\thanks{Department of Radiology, University of Chicago, 5841 S. Maryland Avenue MC2026, Chicago IL, 60637, \texttt{seanrose949@gmail.com}}\and Frank~Sch\"{o}pfer\thanks{Institut f\"ur Mathematik, Carl von Ossietzky Universit\"at Oldenburg, 26111 Oldenburg, Germany, \texttt{frank.schoepfer@uni-oldenburg.de}}}

\maketitle

\begin{abstract}
  This paper investigates the randomized version of the Kaczmarz method to solve linear systems in the case where the adjoint of the system matrix is not exact---a situation we refer to as ``mismatched adjoint''.
  We show that the method may still converge both in the over- and underdetermined consistent case under appropriate conditions, and we calculate the expected asymptotic rate of linear convergence.
  Moreover, we analyze the inconsistent case and obtain results for the method with mismatched adjoint as for the standard method.
  Finally, we derive a method to compute optimized probabilities for the choice of the rows and illustrate our findings with numerical example.
\end{abstract}



\section{Introduction}
\label{sec:intro}

In this paper we consider the solution of linear systems
\begin{equation}
  Ax=b \label{eq:OCS} 
\end{equation}
with row-action methods, i.e methods that only use single rows of the system in each step.
This is beneficial, for example, in situations where the full system is too large to store or keep in memory.
Probably the first method of this type is the Kaczmarz method where each step consists of a projection onto a hyperplane given by the solution space of a single row.
If $a^{T}$ is a row vector of the system and the corresponding entry on the right hand side is (with slight abuse of notation) $b$, then the orthogonal projection of a given vector $x$ onto the solution space of $\scp{a}{x} = b$ is
\[
x - \frac{\scp{a}{x} - b}{\norm{a}^{2}}\cdot a.
\]
Thus, one updates the current vector $x$ in the direction of $a$ which is the corresponding column of $A^{T}$.
A question that has been motivated by the use of the Kaczmarz method in tomographic reconstruction (where it is known under the name \emph{algebraic reconstruction technique} (ART),~\cite{gordon1970algebraic}, see also~\cite{kak2002principles}) is:
\emph{Will the method still converge, if we do not use $A^{T}$ as the adjoint but a different matrix $V^{T}$?}
In tomographic reconstruction, the linear operator $A$ models the ``forward projection'' operation, which maps an object's density to a set of measured line integrals.
The adjoint map $A^{T}$, however, also has a physical interpretation: This map is called ``backprojection'' and, roughly speaking, ``distributes the values along lines through the measurement volume''.
Since both $A$ and $A^{T}$ have their own physical significance, their corresponding maps are often implemented by different means.
For example,~\cite{de2004distance} proposes and discusses several method for the implementation of the backprojection method and shows that special methods compare favorably with respect to reconstruction quality.
In~\cite{zeng2000unmatched}, the authors discuss the use of mismatched projection pairs, for the purposes of improved computational efficiency when using the Landweber algorithm for reconstruction.
Hence, one does not always use the actual adjoint, but a different map and we refer to this situation as using a ``mismatched adjoint''.

The goal of this paper is to analyze the convergence behavior of the randomized Kaczmarz method with mismatched adjoint.

\section{The overdetermined consistent case}

The Kaczmarz method is known to converge for any consistent linear system, but the speed of convergence is hard to quantify since it depends on the ordering of the rows.
This is notably different for the \emph{randomized Kaczmarz method} as shown in~\cite{strohmer2009randomized}: If the rows are chosen independently at random the method converges linearly.
To fix notation, let $A=(a_i^T)_{i=1,\ldots,m} \in \RR^{m \times n}$ with $m \ge n$ and row vectors
$a_i \in \RR^n$ and $V=(v_i^T)_{i=1,\ldots,m}$, with row vectors
$v_i \in \RR^n$.  Moreover let  $p_i>0$,
$i\in \{1,\ldots,m\}$ denote a probability distribution on the set of indices of the rows, i.e., $p_{i}$ is the probability to choose the $i$-th row for the next step.

\begin{algorithm}[h]
  \caption{Randomized Kaczmarz with Mismatched Adjoint}
  \label{alg:RKMA}
  \begin{algorithmic}[1]
    \REQUIRE{starting point $x_0\in\RR^n$ and probabilities $p_i>0$, $i \in \{1,\ldots,m\}$}
    \ENSURE{solution of~\eqref{eq:OCS}}
    \STATE initialize $k = 0$
    \REPEAT
    \STATE choose an index $i_k=i \in \{1,\ldots,m\}$ at random with probability $p_i$
    \STATE update $x_{k+1}=x_k-\tfrac{\scp{a_{i_k}}{x_k}-b_{i_k}}{\scp{a_{i_k}}{v_{i_k}}} \cdot v_{i_k}$  
    \STATE increment $k =  k+1$
    \UNTIL{a stopping criterion is satisfied}
  \end{algorithmic}
\end{algorithm}

The algorithm we consider in this work is the randomized Kaczmarz method with mismatched adjoint, abbreviated RKMA, and is given in Algorithm~\ref{alg:RKMA}.
The difference to the standard randomized Kaczmarz method is that the usual projection step 
$x_{k+1} =
x_{k}-\tfrac{\scp{a_{i_k}}{x_k}-b_{i_k}}{\norm{a_{i_{k}}}^{2}} \cdot
a_{i_k}$
is replaced by
$x_{k+1} =
x_{k}-\tfrac{\scp{a_{i_k}}{x_k}-b_{i_k}}{\scp{a_{i_k}}{v_{i_k}}} \cdot
v_{i_k}$.
This results in $\scp{x_{k+1}}{a_{i_{k}}} = b_{i_{k}}$, i.e., the next iterate
$x_{k+1}$ is on the hyperplane defined by the $i_{k}$-th equation of
the system, but since $v_{i_{k}}$ is not orthogonal to this
hyperplane, this is an \emph{oblique} projection, instead of an
orthogonal projection as it would be in the original Kaczmarz method (see Figure~\ref{fig:oblique-projection}).

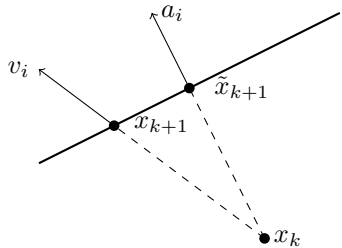
\begin{figure}[htb]
  \centering

\begin{tikzpicture}
  \fill (0,0) circle (2pt) node[right]{$x_{k}$};

  \draw[thick] (-1,2)++(-2,-1) -- (-1,2)--++(2,1);
  
  \fill (-1,2) circle (2pt) node[right=6pt]{$\tilde x_{k+1}$};
  
  \draw[dashed] (0,0) -- (-1,2);
  
  \draw[->] (-1,2)--+(-0.5,1)node[right]{$a_{i}$};
  
  \fill (-2,1.5) circle(2pt) node[right=4pt]{$x_{k+1}$};
  
  \draw[dashed] (0,0) -- (-2,1.5);
  
  \draw[->] (-2,1.5) --++ (-1,0.75) node[left]{$v_{i}$};
\end{tikzpicture}

\caption{Oblique projection $x_{k+1}$ of $x_{k}$ onto the hyperplane $\{x\mid \scp{a_{i}}{x}=b_{i}\}$. The orthogonal projection is $\tilde x_{k+1}$.}
\label{fig:oblique-projection}
\end{figure}



To formulate the convergence theorem for RKMA we denote by $\lambda_{\min}(M)$ the smallest
eigenvalue of a symmetric real matrix $M$.
For general (non-symmetric) real square matrices $M$ we denote by $\rho(M)$ its \emph{spectral radius}, i.e. the largest absolute value of its eigenvalues.

First we state a result on the expected outcome of one step of RKMA.
A similar result has been observed earlier in the case with
no mismatch, see e.g.~\cite{strohmer2009randomized}, \cite[Lemma 2.2]{needell2010noisy}) or~\cite[Lemma 3.6]{zouzias2013randomized}).
In the following, we generally assume that the rows of $A$ and $V$ fulfill $\scp{a_{i}}{v_{i}}\neq 0$ and, without loss of generality, that $\scp{a_{i}}{v_{i}}>0$.
\begin{lem}
  \label{lem:est-one-step}
  Let $\hat x$ fulfill $A\hat x=b$, $x$ be arbitrary and
  $x^{+} = x - \tfrac{\scp{a_{i}}{x}-b_{i}}{\scp{a_{i}}{v_{i}}} \cdot v_{i}$
  be the oblique projection onto the hyperplane
  $\{x\mid \scp{a_{i}}{x}=b_{i}\}$. Further we let $p_{i}>0$, $i=1,\dots,m$, be probabilities and denote
  $D:=\diag\big(\tfrac{p_i}{\scp{a_i}{v_i}}\big)$ and
  $S:=\diag\big(\tfrac{\norm{v_i}^2}{\scp{a_i}{v_i}}\big)$.
  If $i$ is randomly chosen with
  probability $p_{i}$ (i.e. $x^{+}$ is a random variable) then it
  holds that
  \begin{equation}\label{eq:estimated-step}
    \EE(x^{+}-\hat x) = (I - V^{T}DA)(x-\hat x)
  \end{equation}
  and if 
  \begin{equation}
    \lambda := \lambda_{\min}\big(V^{T}DA + A^{T}DV - A^{T}SDA\big)>0  \label{eq:CC}
  \end{equation}
  is fulfilled, it holds that 
  \[
  \EE(\norm{x^{+}-\hat x}^{2})\leq (1-\lambda)\cdot \norm{x-\hat x}^{2}
  \]
  (where both expectations are with respect to the probabilities $p_{i}$).
\end{lem}
\begin{proof}
  Since $b_{i} = \scp{a_{i}}{\hat x}$, the expectation $\EE(x^{+}-\hat x)$ is 
  \begin{align*}
    \EE(x^{+}-\hat x) & = \sum_{i=1}^{m}p_{i}\cdot (x - \frac{\scp{a_{i}}{x}-b_{i}}{\scp{a_{i}}{v_{i}}}\cdot v_{i}) - \hat x\\
    & = x - \sum_{i=1}^{m}p_{i}\cdot \frac{\scp{a_{i}}{x-\hat x}}{\scp{a_{i}}{v_{i}}}\cdot v_{i} - \hat x\\
    & = x - \hat x - \sum_{i=1}^{m}\tfrac{p_{i}}{\scp{a_{i}}{v_{i}}}\cdot v_{i}a_{i}^{T}(x-\hat x),
  \end{align*}
	from which~\eqref{eq:estimated-step} follows.
  To calculate the expectation of the squared norm we calculate
  \begin{align}
    \norm{x^{+}-\hat{x}}^2 = &\norm{x-\hat{x}}^2 - 2 \cdot \frac{\scp{a_{i}}{x-\hat{x}} \cdot \scp{v_{i}}{x-\hat{x}}}{\scp{a_{i}}{v_{i}}} \nonumber \\
                               & + \frac{\big(\scp{a_{i}}{x-\hat x}\big)^2}{\big(\scp{a_{i}}{v_{i}}\big)^2} \cdot \norm{v_{i}}^2 \,. \label{eq:error_exact}
  \end{align}
  Taking the expectation gives
  \begin{align*}
    \EE(\norm{x^{+}-\hat{x}}^2) = & \norm{x-\hat{x}}^2 \\
                                  & - \sum_{i=1}^{m} p_i \cdot 2 \cdot \frac{\scp{a_i}{x-\hat{x}} \cdot \scp{v_i}{x-\hat{x}}}{\scp{a_i}{v_i}} \\
                                  & + \sum_{i=1}^{m} p_i \cdot\frac{\big(\scp{a_i}{x-\hat x}\big)^2}{\big(\scp{a_i}{v_i}\big)^2} \cdot \norm{v_i}^2 \,.
  \end{align*}
  By the definition of $D$ and $S$ the right hand side can be written as
  \begin{align}
    \norm{x-\hat x}^{2} - \scp{x-\hat x}{(2V^{T}DA -
      A^{T}SDA)(x-\hat{x})}\nonumber\\
    = \norm{x-\hat x}^{2} - \scp{x-\hat x}{(2V^{T} -
      A^{T}S)DA(x-\hat{x})}\label{eq:est_rkma-aux}
  \end{align}
  and hence, we aim to bound $\scp{x-\hat x}{(2V^{T} -
    A^{T}S)DA(x-\hat{x})}$ from below. More precisely, we want
  \[
  \scp{x-\hat x}{(2V^{T} - A^{T}S)DA(x-\hat{x})}\geq \lambda\cdot \norm{x-\hat x}^{2}
  \]
  and this is the case if and only if
  \[
  \scp{x-\hat x}{((2V^{T} - A^{T}S)DA - \lambda I)(x-\hat{x})}\geq 0.
  \]
  Since we have $2\scp{z}{V^TDA z} = \scp{z}{(V^TDA+A^TDV) z}$ for all $z$, this is equivalent to
  \[
  \scp{x-\hat x}{(V^{T}DA + A^{T}DV - A^{T}SDA - \lambda I)(x-\hat{x})}\geq 0
  \]
  and this is ensured if
  \[
  \lambda_{\min}(V^{T}DA + A^{T}DV -A^{T}SDA)\geq \lambda.
  \]
  Hence, if~\eqref{eq:CC} is fulfilled, we obtain the estimate
  \[
  \EE(\norm{x_{k+1}-\hat{x}}^2) \le (1- \lambda) \cdot \norm{x-\hat{x}}^2 \,.\qedhere
  \]
\end{proof}
Equation~\eqref{eq:estimated-step} shows that $\norm{\EE(x^{+}-\hat x)}^{2}\leq \norm{I-V^{T}DA}^{2}\norm{x-\hat x}^{2}$. Recall that $\rho(M)\leq \norm{M}$ for asymmetric matrices $M$, and note that the above inequality is not true, if we replace the norm by the spectral radius.
Due to Jensen's inequality we generally have $\norm{\EE(x^{+}-\hat x)}^{2}\leq \EE(\norm{x^{+}-\hat x}^{2})$ and Lemma~\ref{lem:est-one-step} provides different estimates for both quantities.

Iterating the previous lemma, we obtain the convergence result:
\begin{thm}
  \label{thm:RKMA}
  Assume that the assumptions of Lemma~\ref{lem:est-one-step} are
  fulfilled and denote by $x_k$ the iterates of
  Algorithm~\ref{alg:RKMA}.
  
  If $\rho(I-V^{T}DA)<1$ then $x_k$ converges in expectation to $\hat x$,
  \[
  \EE(x_{k}-\hat x) \to 0 \quad \mbox{for} \quad k \to \infty \,,
  \]
  moreover, it holds that
  \[
  \norm{\EE(x_{k}-\hat x)}\leq \norm{I-V^{T}DA}^{k}\norm{x_{0}-\hat x}.
  \]
  If condition \eqref{eq:CC} 
  is fulfilled then it holds that
  \begin{equation*}
    \EE \left[\norm{x_{k+1}-\hat{x}}^2\right] \le (1- \lambda) \cdot \EE
    \left[\norm{x_k-\hat{x}}^2 \right] \,.
  \end{equation*}
\end{thm}
\begin{proof}
  The first claim follows from Lemma~\ref{lem:est-one-step} and the well known fact that $(I-V^{T}DA)^{k}\to 0$ if the spectral radius of $I-V^{T}DA$ is smaller than one (see, e.g.,~\cite[Theorem 11.2.1]{golub2013matrix}). The second claim is also immediate from the previous lemma.
  
  Finally, we get for expectation with respect to $i_{k}$ (conditional on $i_{0},\dots, i_{k-1}$)
  \[
  \EE \left[\norm{x_{k+1}-\hat{x}}^2\,\middle|\, i_0,\ldots,i_{k-1} \right] \le (1- \lambda) \norm{x_k-\hat{x}}^2
  \]
  Now we consider all indices $i_0,\ldots,i_k$ as random variables
  with values in $\{1,\ldots,m\}$, and take the full expectation on
  both sides to get the assertion.
\end{proof}

Here are some remarks on the result:

\begin{rem}
  Since eigenvalues depend continuously on perturbations, both
  condition~\eqref{eq:CC} and $\rho(I-V^{T}DA)<1$ are fulfilled for
  $V \approx A$. Note that $\norm{I-V^{T}DA} = \rho(I-V^{T}DA)$ does
  hold for $V = A$ and is generally not true otherwise. It may even be the case
  that $\norm{I-V^{T}DA}>1$ while $\rho(I-V^{T}DA)<1$.
\end{rem}

\begin{rem}[Relation to the result of Strohmer and Vershynin]\label{rem:strohmer-vershynin}
  Note that Theorem~\ref{thm:RKMA} contains the result of Strohmer and
  Vershynin~\cite{strohmer2009randomized} as a special case: Take
  $V=A$ and the probabilities $p_{i}$ proportional to the squared
  row-norms, i.e. $p_{i} =
  \frac{\norm{a_{i}}^{2}}{\norm[F]{A}^{2}}$. Then we have
  \[
  D=\diag\big(\tfrac{p_i}{\scp{a_i}{v_i}}\big) =
  \tfrac{1}{\norm[F]{A}^{2}}\cdot I\quad \text{and}\quad
  S=\diag\big(\tfrac{\norm{v_i}^2}{\scp{a_i}{v_i}}\big) = I
  \]
  and hence we get
  \[
  \lambda = \frac{\lambda_{\min}(A^{T}A)}{\norm[F]{A}^{2}} =
  \frac{\sigma_{\min}(A)}{\norm[F]{A}^{2}}
  \]
  (where $\sigma_{\min}(A)$ denotes the smallest singular value of $A$) as in~\cite{strohmer2009randomized}.
  
  To get a similarly simple expression for the convergence of the method with mismatch we  set
  \[
  p_{i} =
  \frac{\scp{a_{i}}{v_{i}}}{\norm[V]{A}^{2}},\quad\text{with}\quad\norm[V]{A}^{2}
  = \sum_{i}\scp{a_{i}}{v_{i}}.
  \]
  This leads to
  \[
  D = \tfrac{1}{\norm[V]{A}^{2}}\cdot I
  \]
  and thus, from~\eqref{eq:estimated-step},
  \[
  \norm{\EE(x_{k+1}-\hat x)}\leq \norm{I -
    \tfrac{V^{T}A}{\norm[V]{A}^{2}}} \norm{x_{k}-\hat x} =
  \sigma_{\max}(I -
  \tfrac{V^{T}A}{\norm[V]{A}^{2}})\cdot\norm{x_{k}-\hat x}.
  \]
  However, in general the contraction factor does not simplify to
  $1 - \tfrac{\sigma_{\min}(V^{T}A)}{\norm[V]{A}^{2}}$ as it would in
  the case with no mismatch.
  
  We also get
  \[
  \EE(\norm{x_{k+1}-\hat x}) \leq \left(1 - \tfrac{\lambda_{\min}(V^{T}A + A^{T}V - A^{T}SA)}{\norm[V]{A}^{2}}\right)^{1/2}\norm{x_{k}-\hat x}
  \]
  for the expectation of the error.
\end{rem}

\begin{rem}[Asymptotic convergence rate and expected improvement in norm]
  The above theorem states that the RKMA method has the asymptotic
  convergence rate of 
  \begin{equation}\label{eq:asymp-rate-RKMA}
    \rho(I-V^{T}DA)
  \end{equation}
  (in expectation), however, the
  expected improvement of the squared error, i.e.  $\EE(\norm{x_{k}-\hat x}^{2})$ in every iteration is
  \begin{equation}\label{eq:expected-improvement-RKMA}
    \begin{split}
      (1-\lambda_{\min}(V^{T}DA+A^{T}DV - A^{T}SDA)) & = \rho(I -
      V^{T}DA - A^{T}DV + A^{T}SDA)\\
      & = \norm{I -
      V^{T}DA - A^{T}DV + A^{T}SDA}.
    \end{split}
  \end{equation}
  Using the spectral norm we can also estimate
  \begin{equation*}
    \norm{\EE(x_{k+1}-\hat x)} = \norm{(I-V^{T}DA)(x_{k}-\hat x)} \leq \norm{I - V^{T}DA}\cdot \norm{x_{k}-\hat x}.
  \end{equation*}
  We can express this norm by the spectral radius as
  \begin{equation}\label{eq:est-norm-of-expectation}
    \norm{I-V^{T}DA} = \rho(I - V^{T}DA - A^{T}DV + A^{T}DVV^{T}DA).
  \end{equation}
  Note that all three expressions in~\eqref{eq:asymp-rate-RKMA},~\eqref{eq:expected-improvement-RKMA}~\eqref{eq:est-norm-of-expectation} are equal in the case of $V=A$, but for the mismatched case, they are  in general different.
  Numerically it seems like~$\eqref{eq:asymp-rate-RKMA}\leq~\eqref{eq:est-norm-of-expectation}\leq~\eqref{eq:expected-improvement-RKMA}$, but we do not have a proof for this.

\end{rem}

\begin{rem}[Different possibilities for stepsizes]
  We could consider the slightly more general iteration
  \[
  x_{k+1} = x_{k} -
  \omega_{i_{k}}\cdot (\scp{a_{i_{k}}}{x_{k}}-b_{i_{k}})\cdot v_{i_{k}}
  \]
  with a steplength $\omega_{i_{k}}$. The iteration in
  Algorithm~\ref{alg:RKMA} uses $\omega_{i} = \scp{a_{i}}{v_{i}}^{-1}$, but
  there are other meaningful choices:
  \begin{itemize}
  \item As for the case with no mismatch, one could take
    $\omega_{i_{k}} = \norm{a_{i_{k}}}^{-2}$, but this would not imply
    $\scp{x_{k+1}}{a_{i_{k}}} = b_{i_{k}}$. Similarly, 
    $\omega_{i} = \norm{v_{i_{k}}}^{-2}$ does not imply $\scp{x_{k+1}}{v_{i_{k}}} = b_{i_{k}}$.
    \item The choice $\omega_{i_{k}} = \frac{\scp{x_{k}}{v_{i_{k}}}-b_{i_{k}}}{(\scp{x_{k}}{a_{i_{k}}}-b_{i_{k}})\norm{v_{i}}^{2}}$ implies that  $\scp{x_{k+1}}{v_{i_{k}}} = b_{i_{k}}$.
  \end{itemize}
  Although none of these cases guarantees that the iterates solve one
  of the equations of the linear system $Ax=b$, one can still deduce
  that iterates converge to the solution of this system of equalities.
  The result of Theorem~\ref{thm:RKMA} can also be derived for this
  slightly more general iteration and the respective condition for
  linear convergence with contraction factor $(1-\lambda)$ is that
  \[
  \lambda := \lambda_{\min}\big(V^{T}DA + A^{T}DV - A^{T}SDA\big) > 0 
  \]
  with 
  \[
  D = \diag(p_{i}\omega_{i}),\qquad S = \diag(\omega_{i}\norm{v_{i}}^{2}).
  \]
\end{rem}

Experiments show that other probabilties than
$p_{i} = \norm{a_{i}}^{2}/\norm[F]{A}^{2}$ in the case $V=A$ or $p_{i} = \scp{a_{i}}{v_{i}}/\norm[V]{A}^{2}$ in the mismatched case frequently lead to faster
convergence. This should not be surprising as one could scale the rows
of system $Ax=b$ arbitrarily by multiplying with a diagonal matrix
which leaves the solution unchanged, but leads to arbitrary row-norms
of the scaled system. In this sense, the row-norms do not reflect the geometry of the arrangements of hyperplanes. We will come back to the problem of selecting probabilities
in Section~\ref{sec:optim-probs}.

\section{Inconsistent overdetermined systems}
\label{sec:inconsistent}

Now we consider the inconsistent case, i.e. we do not assume that the
overdetermined system has a solution. This case has been treated
in~\cite{needell2010noisy} for the case $V=A$. We model an additive
error and assume that the right hand side is $b+r$ with $b\in\rg A$.
\begin{thm}
  \label{thm:inconsitent}
  Denote by $\hat x$ the unique solution of $Ax=b$ and let $x_{k}$
  denote the iterates of Algorithm~\ref{alg:RKMA} where the right hand
  side is $b+r$.
  With $M = (I-V^{T}DA)$ it holds that
  \[
  \EE(x_{k}-\hat x) = M^{k}(x_{0}-\hat x) + \sum_{l=0}^{k-1}M^{l}V^{T}Dr.
  \]
  Moreover, with $\lambda$ defined in~\eqref{eq:CC}, we have
  \[
  \EE(\norm{x_{k}-\hat x}^{2})\leq (1-\tfrac{\lambda}{2})^{k}\cdot \norm{x_{0}-\hat x}^{2} + \tfrac2\lambda \cdot \gamma^{2}
  \]
  with $\gamma := \max_{i}\tfrac{|r_{i}|\cdot \norm{v_{i}}}{|\scp{a_{i}}{v_{i}}|}$.
\end{thm}
\begin{proof}
  For the iterate $x_{k}$ we denote by $\tilde x_{k+1}$ the oblique
  projection onto the ``true hyperplane''
  $H = \{x\mid \scp{a_{i}}{x}=b_{i}\}$, i.e.
  $\tilde x_{k+1} = x_{k} - \frac{\scp{a_{i}}{x_{k}} -
    b_{i}}{\scp{v_{i}}{a_{i}}}\cdot v_{i}$. Then it holds that 
  \[
  x_{k+1}-\hat x = \tilde x_{k+1}-\hat x + \tfrac{r_{i}}{\scp{a_{i}}{v_{i}}}\cdot v_{i}.
  \]
  For one step of the method we get (taking the expectation with respect to the random variable $i_{k+1}$)
  \[
  \EE(x_{k+1}-\hat x) = \EE(\tilde x_{k+1}-\hat x) + \EE(\tfrac{r_{i_{k}}}{\scp{a_{i_{k}}}{v_{i_{k}}}}v_{i_{k}}) = (I-V^{T}DA)(x_{k}-\hat x) + V^{T}Dr.
  \]
  The formula for $\EE(x_{k}-\hat x)$ (with the expectation with
  respect to all indices $i_{0},\dots,i_{k}$) follows by induction.
  
  Moreover, we get
  \begin{align*}
    \norm{x_{k+1}-\hat x}^{2} &  = \norm{\tilde x_{k+1}-\hat x}^{2} + 2\tfrac{r_{i}}{\scp{a_{i}}{v_{i}}}\cdot \scp{\tilde x_{k+1}-\hat x}{v_{i}} + \tfrac{r_{i}^{2}}{\scp{a_{i}}{v_{i}}^{2}}\cdot \norm{v_{i}}^{2}\\
    & \leq \norm{\tilde x_{k+1}-\hat x}^{2} + 2\tfrac{r_{i}}{\scp{a_{i}}{v_{i}}}\cdot \scp{\tilde x_{k+1}-\hat x}{v_{i}} + \gamma^{2}.
  \end{align*}
  Now we use Cauchy-Schwarz and Young with $\epsilon>0$ (i.e. $2ab\leq \epsilon a^{2} + b^{2}/\epsilon$) to get
  \begin{align*}
    \norm{x_{k+1}-\hat x}^{2} &\leq \norm{\tilde x_{k+1}-\hat x}^{2} + 2\norm{\tilde x_{k+1}-\hat x}\cdot \tfrac{r_{i}}{\scp{a_{i}}{v_{i}}}\cdot \norm{v_{i}} + \gamma^{2}\\
    & \leq (1+\epsilon)\cdot \norm{\tilde x_{k+1}-\hat x}^{2} + (1+\tfrac1\epsilon)\cdot \gamma^{2}.
  \end{align*}
  Applying Lemma~\ref{lem:est-one-step} we get
  \[
    \EE(\norm{x_{k+1}-\hat x}^{2})  \leq (1+\epsilon)\cdot (1-\lambda)\cdot \norm{x_{k}-\hat x}^{2} + (1+\tfrac1\epsilon)\cdot \gamma^{2}.
  \]
  Recursively we obtain
  \[
  \EE(\norm{x_{k}-\hat x}^{2}) \leq
  \Big((1+\epsilon)\cdot (1-\lambda)\Big)^{k}\cdot \norm{x_{0}-\hat x}^{2} +
  \sum_{j=0}^{k-1}\Big((1+\epsilon)\cdot (1-\lambda)\Big)^{j}\cdot (1+\tfrac1\epsilon)\cdot \gamma^{2}.
  \]
  Now we choose  $\epsilon = \tfrac{\lambda}{2(1-\lambda)}$, observe that
  \[
  (1-\lambda)\cdot (1+\epsilon) = 1-\tfrac{\lambda}2\quad\text{and}\quad (1 + \tfrac1\epsilon) = 1 - \tfrac\lambda2
  \]
  and get
  \[
  \begin{split}
    \EE(\norm{x_{k}-\hat x}^{2}) & \leq
    (1-\tfrac\lambda2)^{k}\cdot \norm{x_{0}-\hat x}^{2} +
    \sum_{j=0}^{k-1}(1-\tfrac\lambda2)^{j+1}\cdot \gamma^{2}\\
    & \leq
    (1-\tfrac\lambda2)^{k}\cdot \norm{x_{0}-\hat x}^{2} +
    \tfrac{2-\lambda}{\lambda}\cdot \gamma^{2}
  \end{split}
  \]
  which proves the claim.
\end{proof}

The first equation in Theorem~\ref{thm:inconsitent} shows that the
iteration of RKMA will reach a final error of the order of
$\norm{\sum_{l=0}^{\infty}M^{l}V^{T}Dr} = \norm{(I-M)^{-1}V^{T}Dr} =
\norm{(V^{T}DA)^{-1}V^{T}Dr}$ if $\rho(M)<1$.

\section{Underdetermined systems}
\label{sec:underdetemined}

Now we consider the underdetermined case, i.e. the case where $m<n$,
but we will still assume full row rank of $A$ and $V$. 
In the case of no mismatch, linear convergence has been proven for the probabilities $p_{i} = \norm{a_{i}}^{2}/\norm[F]{A}^{2}$ in~\cite{liu2014asynchronous}.
In this case,
Theorem~\ref{thm:RKMA} does never ensure convergence: On the one hand, the matrix
$V^{T}DA + A^{T}DV$ is never positive definite, so $\lambda$ from \eqref{eq:CC} is always zero. On the other hand, $V^{T}DA$ always has a non-trivial kernel, and thus, $I - V^{T}DA$ always has a spectral radius equal to one.  However, the iteration
often converges in practice and this is due to the following simple
observation: All the iterates $x_{k}$ of Algorithm~\ref{alg:RKMA} are
in $\rg V^{T}$ if the starting point $x_{0}$ is there. So, if the
equation $Ax = b$ has a solution $\hat x$ in $\rg V^{T}$, then all
vectors $x_{k}-\hat x$ are also in the range.

Inspecting the proof of
Lemma~\ref{lem:est-one-step} we note that the constant $\lambda$
that needs to be positive to guarantee improvement in each step, is in fact not the smallest eigenvalue of
$V^{T}DA + V^{T}DA - A^{T}DSA$ but the smallest eigenvalue of this
matrix when restricted to the range of $V^{T}$. More explicitly, let
$Z\in\RR^{n\times m}$ be a matrix whose columns form on orthonormal basis
of $\rg V^{T}$. So, the term in~\eqref{eq:est_rkma-aux} can also be written as
\begin{align*}
  \norm{x_{k}-\hat x}^{2} - \scp{ZZ^{T}(x_{k}-\hat x)}{(2V^{T}DA - A^{T}DSA)ZZ^{T}(x_{k}-\hat x)}\\
= \norm{x_{k}-\hat x}^{2} - \scp{Z^{T}(x_{k}-\hat x)}{Z^{T}(2V^{T}DA - A^{T}DSA)ZZ^{T}(x_{k}-\hat x)}.
\end{align*}
Consequently, we need an estimate of the form
\[
\scp{Z^{T}(x_{k}-\hat x)}{Z^{T}(2V^{T}DA - A^{T}DSA)ZZ^{T}(x_{k}-\hat x)}\geq \lambda\cdot \norm{x_{k}-\hat x}^{2}
\]
and, since $\norm{Z^{T}(x_{k}-\hat x)}^{2} = \norm{x_{k}-\hat x}^{2}$, this is fulfilled for
\[
\lambda = \lambda_{\min}(Z^{T}(V^{T}DA +A^{T}DV- A^{T}DSA)Z).
\]

Similarly, the convergence of $\EE(x_{k}-\hat x)$ is equivalent to the convergence of $\EE(Z^{T}(x_{k}-\hat x))$, and it holds that
\begin{align*}
\EE(Z^{T}(x_{k+1}-\hat x)) & = Z^{T}(I-V^{T}DA)(x_{k}-\hat x)\\
                          & = Z^{T}(I-V^{T}DA)ZZ^{T}(x_{k}-\hat x)\\
                          & = (I - Z^{T}V^{T}DAZ)Z^{T}(x_{k}-\hat x).
\end{align*}
Finally, note that the system $Ax=b$ has only one solution that lies in $\rg V^{T}$ if $AV^{T}$ is non-singular.

Thus, we have proved the following theorem:

\begin{thm}
  \label{thm:RKMA-underdetermined}
  Consider the consistent system~\eqref{eq:OCS} with
  $A,V\in\RR^{m\times n}$ for $m\leq n$ both with full row rank such that $AV^{T}$ is non-singular.
  Furthermore let the columns of $Z$ be an orthonormal basis for $\rg V^{T}$ and let
  $p\in\RR^{m}$ with $p_{i}\geq 0$ and $\sum_{i}p_{i}=1$ and set
  $D:=\diag\big(\tfrac{p_i}{\scp{a_i}{v_i}}\big)$ and
  $S:=\diag\big(\tfrac{\norm{v_i}^2}{\scp{a_i}{v_i}}\big)$.  Then it holds:
  \begin{enumerate}
  \item The system $Ax=b$ has exactly one solution $\hat x$ that lies in
    $\rg V^{T}$.
  
  \item If $x_{0}\in\rg V^{T}$ and $\rho(I-Z^{T}V^{T}DAZ)<1$,
    then the iterates of Algorithm~\ref{alg:RKMA} fulfill
    \[
    \EE(x_{k}-\hat x)\to 0\quad \mbox{for} \quad k \to \infty .
    \]
  \item 
    If $x_{0}\in\rg V^{T}$ and 
    \begin{equation}
      \lambda := \lambda_{\min}\big(Z^{T}(V^{T}DA + A^{T}DV - A^{T}SDA)Z\big) > 0  \label{eq:CC-underdetermined}
    \end{equation}
    is fulfilled, then it holds that
    \begin{equation*}
      \EE \left[\norm{x_{k+1}-\hat{x}}^2\right] \le (1- \lambda) \cdot \EE
      \left[\norm{x_k-\hat{x}}^2 \right] \,.
    \end{equation*}
  \end{enumerate}

\end{thm}

This result has the following practical implication:
If one can measure the quantity $x$ by linear measurements, encoded by the vectors $a_{i}$, but only has less measurements available than degrees of freedom in $x$, it is beneficial to use a mismatched adjoint $V$ with rows $v_{i}^{T}$ such that the $v_{i}$ are close to the vectors $a_{i}$ (such that the convergence condition is fulfilled), but which also ensure that $x$ is in the range of the vectors $v$.
Mismatched forward/back projection models in CT provide a useful example to illustrate this result.
Forward-projection in CT is often implemented using a ray-tracing algorithm known as Siddon's method~\cite{siddon1985fast}.
This algorithm models line integration and has the benefit of being computationally efficient and amenable to parallelization; 
however, it does not model the finite width of the detector bin.
This can lead to Moire pattern artifacts when using a matched forward/back-projection pair if the image pixel size is smaller than the detector bin size~\cite{de2004distance}. 
Mismatched projector pairs --- in which the backprojection operator models the finite detector bin width --- are often used to avoid these artifacts. We illustrate how RKMA can be used in this manner in Section~\ref{sec:numerics}.

\section{Optimizing the probabilities}
\label{sec:optim-probs}

In the case of exact adjoint, a common choice for the probabilities $p_{i}$ is to use $p_{i} = \norm{a_{i}}^{2}/\norm[F]{A}^2$ which leads to the simple expression $\lambda = \lambda_{\min}(A^{T}A)/\norm[F]{A}^{2}$.
However, numerical experiments show that this vector $p$ of probabilities does not lead to the best performance in practice.
This is of no surprise: For any diagonal matrix $W = \diag(w_{i})$ one can consider the problem $WA x = Wb$ which has different row norms,
while the each Kaczmarz iteration stays exactly the same.
This shows that the choice of probabilities by norms of the rows is in some sense arbitrary.
In~\cite{agaskar2014randomized} the authors proposed a method to find the smallest contraction factor of the method by minimizing the largest eigenvalue of an auxiliary matrix of size $\RR^{n^{2}\times n^{2}}$. Here we present a different method that also works for the case of mismatched adjoint.

Theorem~\ref{thm:RKMA} states that the asymptotic convergence rate is given by $\rho(I - V^{T}DA)$, while the expected improvement in each step is either expressed by
$1-\lambda_{\min}(V^{T}DA + A^{T}DV - A^{T}SDA)$ or $\norm{I-V^{T}DA}$
(recall that $D = \diag(p_{i}/\scp{a_{i}}{v_{i}})$ and $S = \diag(s_{i})$ with $s_{i} = \norm{v_{i}}^{2}/\scp{a_{i}}{v_{i}}$).
One would like to choose $p$ (i.e. $D$) in such a way that these quantities are as small as possible.
Numerically, we observe that the asymptotic rate is indeed quite tight, while the expected improvement is only a loose estimate in the case of mismatched adjoint.
However, the numerical radius of a non-symmetric matrix is not easily characterized and is neither a convex, nor concave function of the entries of the matrix.
The minimal eigenvalue of a symmetric matrix, on the other hand, is characterized by a minimization problem and it will turn out, that $\lambda_{\min}$ is indeed a concave function in $p$. 
Also, the spectral norm is convex and thus, the function $\norm{I-V^{T}DA}$ is also convex in $p$.
We therefore aim to choose $p$ such that $\lambda_{\min}$ is maximized or $\norm{I-V^{T}DA}$ is minimized, i.e. we aim to solve
\begin{equation}
  \label{eq:best-prob-vector}
  \max_{p}\ \lambda_{\min}(V^{T}DA + A^{T}DV - A^{T}SDA),\quad \text{s.t.}\quad \sum_{i=1}^{m}p_{i} = 1,\quad p\geq 0.
\end{equation}
or
\begin{equation}
  \label{eq:best-prob-vector_norm}
  \min_{p}\ \norm{I-V^{T}DA},\quad \text{s.t.}\quad \sum_{i=1}^{m}p_{i} = 1,\quad p\geq 0.
\end{equation}

\subsection{Maximizing $\lambda_{\min}$}
\label{sec:max_lambda_min}

The super-gradient of the objective functional in~\eqref{eq:best-prob-vector} is given by the next lemma:
\begin{lem}
  \label{lem:supergrad-smallest-ev}
  The function $f(p) =  \lambda_{\min}(V^{T}DA + A^{T}DV - A^{T}SDA)$ is concave.
  A super-gradient at $p$ is given by
  \[
  \frac{\partial\lambda_{\min}}{\partial p} = \left(\frac{\scp{2v_{i}-s_{i}a_{i}}{x}\scp{a_{i}}{x}}{\scp{a_{i}}{v_{i}}}\right)_{i=1,\dots,m}
  \]
  where $x$ is an eigenvector of $V^{T}DA + A^{T}DV - A^{T}SDA$ corresponding to the smallest eigenvalue.
\end{lem}
\begin{proof}
  By the min-max principle for eigenvalues of symmetric matrices, we have
  \[
  \begin{split}
    \lambda_{\min}(V^{T}DA + A^{T}DV - A^{T}SDA) & = \min_{\norm{x}=1}\scp{(V^{T}DA + A^{T}DV - A^{T}SDA)x}{x}\\
    & = \min_{\norm{x}=1}\scp{DAx}{(2V - SA)x}\\
    & = \min_{\norm{x}=1}\sum_{i=1}^{m}p_{i}\frac{\scp{2v_{i}-s_{i}a_{i}}{x}\scp{a_{i}}{x}}{\scp{a_{i}}{v_{i}}}.
  \end{split}
  \]
  This shows that $f$ is a minimum over linear functions in $p$, and hence, concave.
  
  To compute a super-gradient, let $x$ be a minimizer, i.e. an
  eigenvector of $V^{T}DA + A^{T}DV - A^{T}SDA$ corresponding to the
  smallest eigenvalue. Since this is a point where the minimum is
  assumed, a super-gradient is given by
  \[
  \frac{\partial\lambda_{\min}}{\partial p} = \left(\frac{\scp{2v_{i}-s_{i}a_{i}}{x}\scp{a_{i}}{x}}{\scp{a_{i}}{v_{i}}}\right)_{i=1,\dots,m}.
  \]
\end{proof}

The previous lemma allows one to solve~\eqref{eq:best-prob-vector} by projected super-gradient ascent as follows: Choose a stepsize sequence $t_{k}$ and iterate:
\begin{enumerate}
\item Initialize with $p^{0}_{i}= 1/m$, $k=0$
\item Form $V^{T}DA + A^{T}DV - A^{T}SDA$ and compute an eigenvector $x$ corresponding to the minimal eigenvalue.
\item Compute the super-gradient $g^{k}_{i} =  \frac{\partial\lambda_{\min}}{\partial p}(p^{k})$ according to Lemma~\ref{lem:supergrad-smallest-ev}.
\item Update $p^{k+1} = \proj_{\Delta_{m}}(p^{k} + t_{k}g^{k})$ where $\proj_{\Delta_{m}}$ is the projection onto the $m$-dimensional simplex.
\end{enumerate}

It is worth noting, how this algorithm looks in the special case of $V=A$.
There we only want to maximize $\lambda_{\min}(A^{T}DA)$ and the super-gradient of this at some $p^{k}$ is just $g^{k}=\left(\tfrac{\scp{a_{i}}{x}^{2}}{\norm{a_{i}}^{2}}\right)_{i=1m\dots,m}$.
As this is always positive, we can project onto the simplex by a simple rescaling as
\[
p^{k+1} = \frac{p^{k}+ t_{k}g^{k}}{\norm[1]{p^{k}+t_{k}g^{k}}}.
\]

\subsection{Minimizing $\norm{I-V^TDA}$}
\label{sec:min_norm}

The subgradient of the objective functional in~\eqref{eq:best-prob-vector_norm} is given by the next lemma:
\begin{lem}\label{lem:subgrad-spectral-norm}
  Let $s_{i} = \scp{a_{i}}{v_{i}}$.
  The function $f(p) = \norm{I-V^{T}DA}$ with $D = \diag(p/s)$ is convex. A subgradient is given by
  \[
  -\frac{(Aq_{1})\odot(Vr_{1})}{s}\in \partial f(p)
  \]
  where $q_{1}$ and $r_{1}$ are left and right singular values of $I-V^{T}DA$ corresponding the largest singular value, $\odot$ denotes the componentswise product and the division is also to be understood componentwise.
\end{lem}
\begin{proof}
  The convexity of $f$ follows from the convexity of the norm and the
  fact that the map $M: \RR^{m}\to\RR^{n\times n}$, $p\mapsto -V^{T}DA$ is linear in $p$.
  
  Example 1 in \cite{watson1992characterization} shows that the
  subgradient of the spectral norm is given as follows: If $B$ has a singular value decomposition $B = Q\Sigma R^{T}$ and the maximal singular value has multiplicity $j$, then 
  \[
  \partial_{A}\norm{A} = \conv\{q_{i}r_{i}^{T}\mid i=1,\dots,j\}
  \]
  where $q_{i}$ and $r_{i}$ are the $i$ columns of $Q$ and $R$, respectively.
  
  By the chain rule for subgradients, we get that
  \[
  \partial_{p}f(p) = M^{T}\partial\norm{I-Mp}.
  \]
  The transpose of $M$ is calculated by
  \begin{align*}
    \scp{p}{M^{T}B} & = \scp{Mp}{B}\\
    & = \trace((Mp)^{T}B)\\
    & = -\trace(V^{T}\diag(p/s)AB)\\
    & = -\trace(\diag(p/s)ABV^{T})\\
    & = \scp{p}{-\diag(\diag(1/s)ABV^{T})},
  \end{align*}
  i.e.
  \[
  M^{T}B = -\diag\Big(\diag(1/s)ABV^{T}\Big).
  \]
  Plugging in the previous formula we obtain that 
  \[
  -\diag(\diag(1/s)Aq_{1}r_{1}^{T}V^{T}) = -\diag((Vr_{1})^{T}\diag(1/s)(Aq_{1})) = \frac{(Aq_{1})\odot(Vr_{1})}{s} 
  \]
  is a subgradient of $f$,
  which shows the assertion.
\end{proof}

Similar to the previous subsection we can solve~\eqref{eq:best-prob-vector_norm} by projected subgradient descent as follows: Choose a stepsize sequence $t_{k}$ and iterate:
\begin{enumerate}
\item Initialize with $p^{0}_{i}= 1/m$, $k=0$
\item Compute a pair $q,r$ of left and right singular vectors of $I-V^{T}DA$.
\item Compute a subgradient $g^{k}$ according to Lemma~\ref{lem:subgrad-spectral-norm}
\item Update $p^{k+1} = \proj_{\Delta_{m}}(p^{k} + t_{k}g^{k})$ where $\proj_{\Delta_{m}}$ is the projection onto the $m$-dimensional simplex.
\end{enumerate}

\section{Numerical experiments}
\label{sec:numerics}

In this section we report a few numerical experiments that illustrate the results.\footnote{The code to produce the figures in this article is available at \url{https://github.com/dirloren/rkma}.}
We start with an illustration of Theorem~\ref{thm:RKMA}, i.e. the consistent and overdetermined case.
We used a Gaussian matrix $A\in\RR^{500\times 200}$ (i.e. the entries and independently and normally distributed) and defined the mismatched adjoint $V$ by setting all entries of $A$ with magnitude smaller that $0.5$ to zero.
The unique solution $\hat x$ was also generated as a Gaussian vector and as probabilties we used $p_{i} = \norm{a_{i}}^{2}/\norm[F]{A}^{2}$. The convergence condition~\eqref{eq:CC} is fulfilled and $1-\lambda\approx 1-5.5\cdot 10^{-4}$ and we also have $\rho(I-V^{T}DA) \approx 1 - 7.5\cdot 10^{-4}$.
Figure~\ref{fig:convergence_perturbation} shows the error and the residuals for the randomized Kaczmarz method with and without mismatched adjoint.

\begin{figure}[htb]
  \centering
  \setlength\figurewidth{0.35\textwidth}
%
%
\definecolor{mycolor1}{rgb}{0.00000,0.44700,0.74100}%
\definecolor{mycolor2}{rgb}{0.85000,0.32500,0.09800}%
\definecolor{mycolor3}{rgb}{0.92900,0.69400,0.12500}%
\begin{tikzpicture}

\begin{axis}[%
width=0.951\figurewidth,
height=0.75\figurewidth,
at={(0\figurewidth,0\figurewidth)},
scale only axis,
xmin=0,
xmax=20000,
xlabel style={font=\color{white!15!black}},
xlabel={$k$},
ymode=log,
ymin=1e-08,
ymax=100,
yminorticks=true,
ylabel style={font=\color{white!15!black}},
ylabel={$\|x_k-\hat x\|$},
axis background/.style={fill=white},
legend style={legend cell align=left, align=left, draw=white!15!black},
legend style={font=\scriptsize},
]
\addplot [color=mycolor1]
  table[row sep=crcr]{%
0	14.1672052673367\\
500	5.88312641165113\\
1000	2.91376439783347\\
1500	1.5001905920962\\
2000	0.902082033010813\\
2500	0.538471306589639\\
3000	0.33778803376405\\
3500	0.202825877566636\\
4000	0.131407298650567\\
4500	0.0845190739374194\\
5000	0.0544835611739526\\
5500	0.0343703128023273\\
6000	0.0220052600216057\\
6500	0.0138720785960084\\
7000	0.00907390153526532\\
7500	0.00597108039444771\\
8000	0.00399893854095585\\
8500	0.00249634008292328\\
9000	0.00169841023818331\\
9500	0.00120776504307857\\
10000	0.000852017211629835\\
10500	0.000572829247754201\\
11000	0.000371163944585592\\
11500	0.00022243465564903\\
12000	0.000140846848454149\\
12500	9.44410034841882e-05\\
13000	6.08075160406215e-05\\
13500	3.92210501271866e-05\\
14000	2.43321477499175e-05\\
14500	1.58509739877588e-05\\
15000	1.01625597412349e-05\\
15500	6.84690943367485e-06\\
16000	4.65983392542749e-06\\
16500	3.06786460475631e-06\\
17000	2.00370935954397e-06\\
17500	1.3218276799006e-06\\
18000	8.509970428052e-07\\
18500	5.83955179399466e-07\\
19000	3.54339873668716e-07\\
19500	2.43469470250286e-07\\
20000	1.57377851359516e-07\\
};
\addlegendentry{RK}

\addplot [color=mycolor2]
  table[row sep=crcr]{%
0	14.1672052673367\\
500	6.049107604801\\
1000	2.98499182314685\\
1500	1.43917926277164\\
2000	0.845049083137031\\
2500	0.505868079321398\\
3000	0.290928042005495\\
3500	0.167715909244177\\
4000	0.111210340215367\\
4500	0.0721968526473898\\
5000	0.0502888634001223\\
5500	0.0339181623612032\\
6000	0.0215186433814277\\
6500	0.0138875992301323\\
7000	0.0100037735231531\\
7500	0.00671552740558127\\
8000	0.00453742966021709\\
8500	0.00298086386350287\\
9000	0.0021670379471908\\
9500	0.00135994761926167\\
10000	0.000941052590793615\\
10500	0.000619492876189757\\
11000	0.000431157692401353\\
11500	0.000279087396221911\\
12000	0.000194174152034458\\
12500	0.000134510855794867\\
13000	9.25980996235986e-05\\
13500	5.70677312523665e-05\\
14000	3.54594268621814e-05\\
14500	2.44117480717027e-05\\
15000	1.68638381596283e-05\\
15500	1.16827930260993e-05\\
16000	7.6422178457208e-06\\
16500	4.65535777820914e-06\\
17000	3.12751555633827e-06\\
17500	2.11638979249376e-06\\
18000	1.3859764035559e-06\\
18500	8.58999030820027e-07\\
19000	5.92056060027892e-07\\
19500	3.5957497005796e-07\\
20000	2.45380518391142e-07\\
};
\addlegendentry{RKMA}

\addplot [color=mycolor3]
  table[row sep=crcr]{%
0	14.1672052673367\\
500	9.81214782381784\\
1000	6.79585303520859\\
1500	4.70677972910768\\
2000	3.25989619015635\\
2500	2.25779063015778\\
3000	1.56373646038826\\
3500	1.08303741050458\\
4000	0.75010723498846\\
4500	0.519521171221491\\
5000	0.359818216326766\\
5500	0.249208609720696\\
6000	0.172600853266756\\
6500	0.119542637719463\\
7000	0.0827947368883563\\
7500	0.0573432926291876\\
8000	0.0397157275104413\\
8500	0.0275069487530746\\
9000	0.0190511990370906\\
9500	0.0131947817262092\\
10000	0.00913865129766094\\
10500	0.00632939212433893\\
11000	0.00438371082983519\\
11500	0.00303614000556514\\
12000	0.0021028180213565\\
12500	0.0014564030719389\\
13000	0.00100869874920739\\
13500	0.000698620585369955\\
14000	0.000483861730458348\\
14500	0.000335120635012734\\
15000	0.000232103166962499\\
15500	0.000160753694298696\\
16000	0.000111337344375201\\
16500	7.71117846255473e-05\\
17000	5.340730337611e-05\\
17500	3.69896776187808e-05\\
18000	2.56188978631969e-05\\
18500	1.77435427929135e-05\\
19000	1.22891044152306e-05\\
19500	8.51138293468401e-06\\
20000	5.89494864825518e-06\\
};
\addlegendentry{expected rate}

\end{axis}
\end{tikzpicture}%
%
%
\definecolor{mycolor1}{rgb}{0.00000,0.44700,0.74100}%
\definecolor{mycolor2}{rgb}{0.85000,0.32500,0.09800}%
\begin{tikzpicture}

\begin{axis}[%
width=0.951\figurewidth,
height=0.75\figurewidth,
at={(0\figurewidth,0\figurewidth)},
scale only axis,
xmin=0,
xmax=20000,
xlabel style={font=\color{white!15!black}},
xlabel={$k$},
ymode=log,
ymin=1e-06,
ymax=10000,
yminorticks=true,
ylabel style={font=\color{white!15!black}},
ylabel={$\|Ax_k-b\|$},
axis background/.style={fill=white},
legend style={legend cell align=left, align=left, draw=white!15!black},
legend style={font=\scriptsize},
]
\addplot [color=mycolor1]
  table[row sep=crcr]{%
0	299.156736759464\\
500	83.8065611148295\\
1000	38.0298312935831\\
1500	22.8288860537529\\
2000	13.3188474255175\\
2500	8.34502311204884\\
3000	5.4871921473515\\
3500	3.59484124933501\\
4000	2.32725034245147\\
4500	1.50362872597102\\
5000	0.934825875107607\\
5500	0.628296660751603\\
6000	0.404699435411643\\
6500	0.26561121019421\\
7000	0.182561629044839\\
7500	0.118422163244\\
8000	0.0776161278297844\\
8500	0.0565995284900799\\
9000	0.041080721702944\\
9500	0.0252215813725731\\
10000	0.0174179709216545\\
10500	0.0122234910495785\\
11000	0.00855251550715836\\
11500	0.00577915177025728\\
12000	0.00399384441757269\\
12500	0.00252221782779137\\
13000	0.00180684861438707\\
13500	0.00123592460938548\\
14000	0.000852995758842721\\
14500	0.000584785319683792\\
15000	0.000418049811207177\\
15500	0.000287589260670429\\
16000	0.000190685108761475\\
16500	0.000124292374011254\\
17000	8.37233053666293e-05\\
17500	5.84459828318591e-05\\
18000	3.91608326128044e-05\\
18500	2.58631705611371e-05\\
19000	1.75275256286332e-05\\
19500	1.22935090806211e-05\\
20000	8.46678614177235e-06\\
};
\addlegendentry{RK}

\addplot [color=mycolor2]
  table[row sep=crcr]{%
0	299.156736759464\\
500	106.151816526849\\
1000	45.3651328696274\\
1500	23.0975444744677\\
2000	14.4595625782819\\
2500	8.19229508899233\\
3000	5.44645861624507\\
3500	3.38824669608456\\
4000	2.16043649661779\\
4500	1.55394399819855\\
5000	0.948005785629732\\
5500	0.600260847056577\\
6000	0.400802802518352\\
6500	0.247330806236085\\
7000	0.155697976643773\\
7500	0.107155628862616\\
8000	0.0720626787892393\\
8500	0.0501579233487291\\
9000	0.0319330442782188\\
9500	0.0192642011325274\\
10000	0.0139329323560009\\
10500	0.0083530791248863\\
11000	0.00498119223536771\\
11500	0.00335717672853031\\
12000	0.0020683454178345\\
12500	0.00134249047485722\\
13000	0.000855134855468356\\
13500	0.000572384671517491\\
14000	0.000405418303204034\\
14500	0.000260209974320502\\
15000	0.000156205538101069\\
15500	0.00011657492840931\\
16000	7.65582179662608e-05\\
16500	5.30296203284895e-05\\
17000	3.44231815355275e-05\\
17500	2.08393474624265e-05\\
18000	1.3463202906988e-05\\
18500	8.64643146594766e-06\\
19000	5.99446055640065e-06\\
19500	4.46625920236769e-06\\
20000	2.90606140872311e-06\\
};
\addlegendentry{RKMA}

\end{axis}
\end{tikzpicture}%
  \caption{Comparison of the randomized Kaczmarz method with and without mismatched adjoint in the overdetermined and consistent case. Left: Decay of the error and also the upper bound from Theorem~\ref{thm:RKMA}. Right: Decay of the residual.}
  \label{fig:convergence_perturbation}
\end{figure}
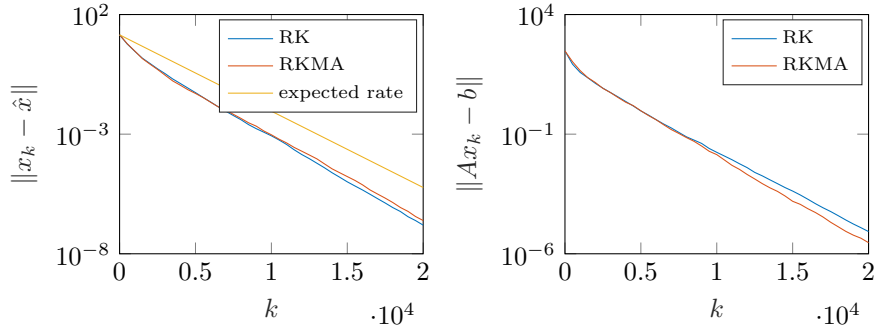

As expected, both methods converge, but in this example the method with mismatch converges slightly faster.
We note that this is not universal: other random instances constructed in the same way show different behaviour, although both methods are always quite close to each other.
Another observation is that the upper bound derived from the convergence factor $1-\lambda$ is quite far from the actual behavior.

Our second numerical example treats the inconsistent and overdetermined case.
The matrix $A$ and solution $\hat x$ and the probabilities are similar to the previous example, but now the right hand side is $b = A\hat x + r$ (with Gaussian $r$).
Figure~\ref{fig:inconsistent} shows the result of the RKMA method on this example and also the error bound from Theorem~\ref{thm:inconsitent}.
As predicted, the error does not go to zero, but levels out at a non-zero level (the same is true for the residual).
As in the previous example one sees that the upper bound from Theorem~\ref{thm:inconsitent} is quite loose.

\begin{figure}[htb]
  \centering
  \setlength\figurewidth{0.7\textwidth}
%
%
\definecolor{mycolor1}{rgb}{0.00000,0.44700,0.74100}%
\definecolor{mycolor2}{rgb}{0.85000,0.32500,0.09800}%
\definecolor{mycolor3}{rgb}{0.92900,0.69400,0.12500}%
\begin{tikzpicture}

\begin{axis}[%
width=0.956\figurewidth,
height=0.709\figurewidth,
at={(0\figurewidth,0\figurewidth)},
scale only axis,
xmin=0,
xmax=40000,
xlabel style={font=\color{white!15!black}},
xlabel={$k$},
ymode=log,
ymin=0.001,
ymax=100,
yminorticks=true,
ylabel style={font=\color{white!15!black}},
ylabel={$\|x_k-\hat x\|$},
axis background/.style={fill=white},
legend style={legend cell align=left, align=left, draw=white!15!black},
legend style={font=\scriptsize},
]
\addplot [color=mycolor1]
  table[row sep=crcr]{%
0	13.3422849367521\\
500	12.1584180288107\\
1000	11.0796203814749\\
1500	10.0965693608885\\
2000	9.20076974555281\\
2500	8.38448031321176\\
3000	7.64064694354613\\
3500	6.96284165856482\\
4000	6.34520707389181\\
4500	5.78240578090405\\
5000	5.26957422228322\\
5500	4.80228066237071\\
6000	4.37648688909558\\
6500	3.98851331648514\\
7000	3.63500718614463\\
7500	3.31291359286022\\
8000	3.01944908386808\\
8500	2.75207760355067\\
9000	2.50848857556172\\
9500	2.28657693281389\\
10000	2.08442492254731\\
10500	1.90028552897233\\
11000	1.7325673698753\\
11500	1.57982093620605\\
12000	1.44072605513241\\
12500	1.31408046744276\\
13000	1.19878941958138\\
13500	1.09385617908535\\
14000	0.998373389817121\\
14500	0.911515190208851\\
15000	0.832530023803452\\
15500	0.760734076742405\\
16000	0.695505281565785\\
16500	0.636277830823155\\
17000	0.582537147637106\\
17500	0.533815263646668\\
18000	0.489686557876678\\
18500	0.449763813299824\\
19000	0.413694551540138\\
19500	0.381157610758738\\
20000	0.351859937768631\\
20500	0.325533573319703\\
21000	0.301932819568552\\
21500	0.280831590882706\\
22000	0.262020962536087\\
22500	0.24530694485978\\
23000	0.230508520447739\\
23500	0.217455985955719\\
24000	0.205989634991119\\
24500	0.195958803067929\\
25000	0.187221270562106\\
25500	0.179642988964306\\
26000	0.173098065742432\\
26500	0.16746892067197\\
27000	0.16264651684154\\
27500	0.158530574606244\\
28000	0.155029694599665\\
28500	0.15206134166876\\
29000	0.149551669316542\\
29500	0.14743518873177\\
30000	0.145654304482364\\
30500	0.144158749369618\\
31000	0.14290495445723\\
31500	0.141855388581862\\
32000	0.140977896687402\\
32500	0.140245059863074\\
33000	0.139633593317971\\
33500	0.139123792515066\\
34000	0.138699032725511\\
34500	0.13834532346097\\
35000	0.138050916527265\\
35500	0.137805964653942\\
36000	0.137602226598784\\
36500	0.137432814117681\\
37000	0.137291976067738\\
37500	0.137174915045405\\
38000	0.137077632253054\\
38500	0.136996796664958\\
39000	0.136929634977068\\
39500	0.13687383924116\\
40000	0.136827489482099\\
};
\addlegendentry{upper bound}

\addplot [color=mycolor2]
  table[row sep=crcr]{%
0	13.3415856509871\\
500	4.80752268142071\\
1000	2.56995806271741\\
1500	1.55133112104778\\
2000	0.911725173104285\\
2500	0.524274397124014\\
3000	0.316687260483066\\
3500	0.195750032426726\\
4000	0.127336837057474\\
4500	0.0820874390499173\\
5000	0.0548592160815942\\
5500	0.0370405028825112\\
6000	0.0263716779438311\\
6500	0.0193055940453536\\
7000	0.0159688642317189\\
7500	0.0140322948351573\\
8000	0.0125223378312629\\
8500	0.0125196853276445\\
9000	0.0127274533527299\\
9500	0.0126390207662409\\
10000	0.012052296580614\\
10500	0.0118409337869874\\
11000	0.0121717617695793\\
11500	0.0118103735750914\\
12000	0.0121439510955108\\
12500	0.0116337607114076\\
13000	0.0119216456300045\\
13500	0.0123655639437584\\
14000	0.0115692186061228\\
14500	0.0115764551631339\\
15000	0.01245738342335\\
15500	0.0125309885135332\\
16000	0.0127898907454865\\
16500	0.0123630090780683\\
17000	0.012463349255891\\
17500	0.0126561069804863\\
18000	0.0121840454512158\\
18500	0.0116277176375999\\
19000	0.011246184215981\\
19500	0.0119092962913164\\
20000	0.0118012855776049\\
20500	0.0126263259367474\\
21000	0.0115850719894962\\
21500	0.0113949905760962\\
22000	0.0123770396404581\\
22500	0.0126643877318743\\
23000	0.0113991971905687\\
23500	0.012256374052498\\
24000	0.0113051845350752\\
24500	0.0121748055586696\\
25000	0.0113273777304379\\
25500	0.0115004312634764\\
26000	0.0119771205668975\\
26500	0.0122139234968527\\
27000	0.0125943990461265\\
27500	0.0122920968227781\\
28000	0.0120110355738388\\
28500	0.012438439588075\\
29000	0.0126592314028726\\
29500	0.0120907187394944\\
30000	0.0121589736045471\\
30500	0.0114985977200852\\
31000	0.0121904522133433\\
31500	0.011891661717136\\
32000	0.0119718851670719\\
32500	0.0117218206758311\\
33000	0.0112212910840294\\
33500	0.0116979848443391\\
34000	0.0114276915918759\\
34500	0.0124782047718303\\
35000	0.0121034087093206\\
35500	0.0117386990357442\\
36000	0.0116537254589141\\
36500	0.0115344029248946\\
37000	0.0108653204455565\\
37500	0.0118867351982588\\
38000	0.0121194893984968\\
38500	0.0120183703403075\\
39000	0.011972846367752\\
39500	0.0121464317154578\\
40000	0.0118667518068816\\
};
\addlegendentry{RKMA}

\addplot [color=mycolor3]
  table[row sep=crcr]{%
0	0.00916481947626831\\
500	0.00916481947626831\\
1000	0.00916481947626831\\
1500	0.00916481947626831\\
2000	0.00916481947626831\\
2500	0.00916481947626831\\
3000	0.00916481947626831\\
3500	0.00916481947626831\\
4000	0.00916481947626831\\
4500	0.00916481947626831\\
5000	0.00916481947626831\\
5500	0.00916481947626831\\
6000	0.00916481947626831\\
6500	0.00916481947626831\\
7000	0.00916481947626831\\
7500	0.00916481947626831\\
8000	0.00916481947626831\\
8500	0.00916481947626831\\
9000	0.00916481947626831\\
9500	0.00916481947626831\\
10000	0.00916481947626831\\
10500	0.00916481947626831\\
11000	0.00916481947626831\\
11500	0.00916481947626831\\
12000	0.00916481947626831\\
12500	0.00916481947626831\\
13000	0.00916481947626831\\
13500	0.00916481947626831\\
14000	0.00916481947626831\\
14500	0.00916481947626831\\
15000	0.00916481947626831\\
15500	0.00916481947626831\\
16000	0.00916481947626831\\
16500	0.00916481947626831\\
17000	0.00916481947626831\\
17500	0.00916481947626831\\
18000	0.00916481947626831\\
18500	0.00916481947626831\\
19000	0.00916481947626831\\
19500	0.00916481947626831\\
20000	0.00916481947626831\\
20500	0.00916481947626831\\
21000	0.00916481947626831\\
21500	0.00916481947626831\\
22000	0.00916481947626831\\
22500	0.00916481947626831\\
23000	0.00916481947626831\\
23500	0.00916481947626831\\
24000	0.00916481947626831\\
24500	0.00916481947626831\\
25000	0.00916481947626831\\
25500	0.00916481947626831\\
26000	0.00916481947626831\\
26500	0.00916481947626831\\
27000	0.00916481947626831\\
27500	0.00916481947626831\\
28000	0.00916481947626831\\
28500	0.00916481947626831\\
29000	0.00916481947626831\\
29500	0.00916481947626831\\
30000	0.00916481947626831\\
30500	0.00916481947626831\\
31000	0.00916481947626831\\
31500	0.00916481947626831\\
32000	0.00916481947626831\\
32500	0.00916481947626831\\
33000	0.00916481947626831\\
33500	0.00916481947626831\\
34000	0.00916481947626831\\
34500	0.00916481947626831\\
35000	0.00916481947626831\\
35500	0.00916481947626831\\
36000	0.00916481947626831\\
36500	0.00916481947626831\\
37000	0.00916481947626831\\
37500	0.00916481947626831\\
38000	0.00916481947626831\\
38500	0.00916481947626831\\
39000	0.00916481947626831\\
39500	0.00916481947626831\\
40000	0.00916481947626831\\
};
\addlegendentry{expected final error}

\end{axis}
\end{tikzpicture}%
  \caption{The RKMA method in the inconsistent. The plot shows the decay of the error and the theoretical upper bound and the expected final error from Theorem~\ref{thm:inconsitent}.}
  \label{fig:inconsistent}
\end{figure}
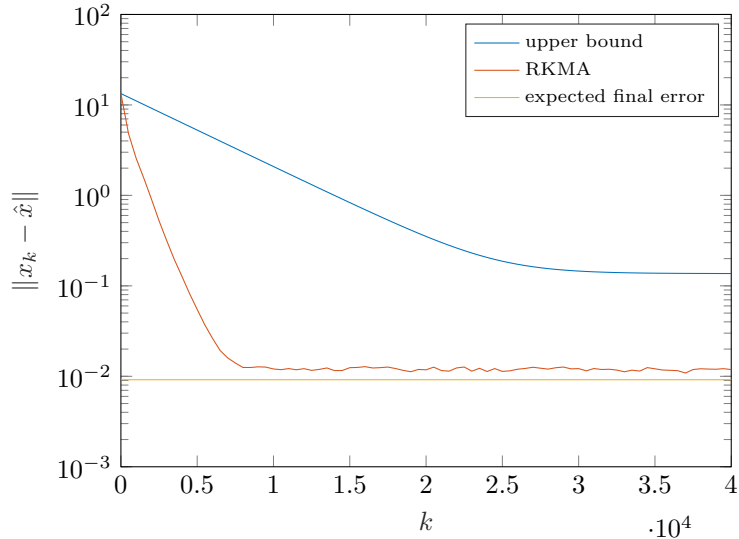

Now we illustrate the behavior of RKMA in the underdetermined case.
We used $A,V\in\RR^{100\times 500}$, again $A$ with Gaussian entries and we obtained $V$ from $A$ by setting the entries of $A$ to zero that have magnitude smaller than $0.3$.
The solution $\hat x$ was constructed as $\hat x = V^{T}c$ for some random vector $c$ and the right hand side was obtained through $b = A\hat x$.
Hence, generically $\hat x$ is not in the range of $A^{T}$ and the standard randomized Kaczmarz method can not converge to $\hat x$.
Figure~\ref{fig:underdetermined} shows that the error decays quickly to zero for RKMA but not for the standard randomized Kaczmarz method. The residuals, however, behave roughly similar for both methods.

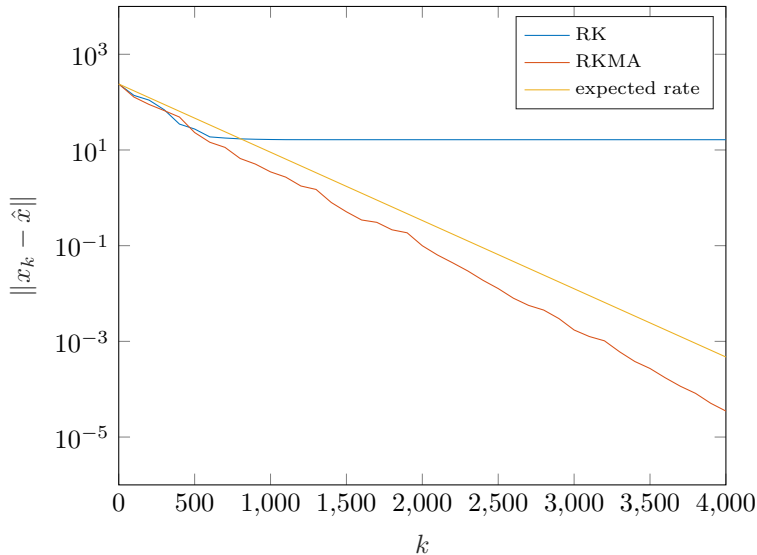
\begin{figure}[htb]
  \centering
  \setlength\figurewidth{0.7\textwidth}
%
%
\definecolor{mycolor1}{rgb}{0.00000,0.44700,0.74100}%
\definecolor{mycolor2}{rgb}{0.85000,0.32500,0.09800}%
\definecolor{mycolor3}{rgb}{0.92900,0.69400,0.12500}%
\begin{tikzpicture}

\begin{axis}[%
width=0.951\figurewidth,
height=0.75\figurewidth,
at={(0\figurewidth,0\figurewidth)},
scale only axis,
xmin=0,
xmax=4000,
xlabel style={font=\color{white!15!black}},
xlabel={$k$},
ymode=log,
ymin=1e-06,
ymax=10000,
yminorticks=true,
ylabel style={font=\color{white!15!black}},
ylabel={$\|x_k-\hat x\|$},
axis background/.style={fill=white},
legend style={legend cell align=left, align=left, draw=white!15!black},
legend style={font=\scriptsize},
]
\addplot [color=mycolor1]
  table[row sep=crcr]{%
0	239.395406768872\\
100	138.357260449599\\
200	110.097240085806\\
300	69.1415599332009\\
400	34.7551955308525\\
500	27.4098855532079\\
600	18.7749136428352\\
700	17.7667544174548\\
800	17.0717815228075\\
900	16.7540983102695\\
1000	16.5527009170385\\
1100	16.4309295540297\\
1200	16.4118176719995\\
1300	16.4046236933987\\
1400	16.3978908603884\\
1500	16.3964233913836\\
1600	16.393741224474\\
1700	16.3932779432323\\
1800	16.3928572639905\\
1900	16.3927757002131\\
2000	16.3927168762873\\
2100	16.3926964332773\\
2200	16.3926895154766\\
2300	16.3926865931491\\
2400	16.3926811300613\\
2500	16.3926798050417\\
2600	16.3926794912642\\
2700	16.3926792023204\\
2800	16.3926789954122\\
2900	16.3926787981092\\
3000	16.3926787414989\\
3100	16.3926787125647\\
3200	16.392678598935\\
3300	16.3926785876061\\
3400	16.3926785801474\\
3500	16.3926785785591\\
3600	16.3926785774149\\
3700	16.3926785771916\\
3800	16.3926785768079\\
3900	16.3926785766695\\
4000	16.3926785759492\\
};
\addlegendentry{RK}

\addplot [color=mycolor2]
  table[row sep=crcr]{%
0	239.395406768872\\
100	127.717229309416\\
200	89.3026121258488\\
300	66.4582498062623\\
400	48.7750431241462\\
500	23.1468313915923\\
600	14.5019057042719\\
700	11.3285051598363\\
800	6.63658737325054\\
900	5.07615625670248\\
1000	3.48523849393842\\
1100	2.70898852884228\\
1200	1.76651290660488\\
1300	1.49671137279737\\
1400	0.79954010186503\\
1500	0.510076228218172\\
1600	0.344148973820167\\
1700	0.307222471415743\\
1800	0.214619023339446\\
1900	0.185762161457548\\
2000	0.0995815658191001\\
2100	0.0638272467397188\\
2200	0.0437955780670308\\
2300	0.0295635566410737\\
2400	0.0188585216056598\\
2500	0.0126829473369861\\
2600	0.00798170986496806\\
2700	0.00564721256306467\\
2800	0.00450322453262679\\
2900	0.00299073915905901\\
3000	0.00172780750380965\\
3100	0.00126284169610775\\
3200	0.00102581586185441\\
3300	0.000603917052341952\\
3400	0.000379353749669092\\
3500	0.000271057253437689\\
3600	0.000173508034554104\\
3700	0.000115314007677437\\
3800	8.18489843737224e-05\\
3900	5.08260757697908e-05\\
4000	3.49311004056378e-05\\
};
\addlegendentry{RKMA}

\addplot [color=mycolor3]
  table[row sep=crcr]{%
0	239.395406768872\\
100	172.369600200995\\
200	124.109645521045\\
300	89.3614888785384\\
400	64.342103799137\\
500	46.3276336736728\\
600	33.356845907653\\
700	24.0176128300553\\
800	17.2931735707689\\
900	12.4514394608988\\
1000	8.96529165187434\\
1100	6.45519376740122\\
1200	4.64787183649334\\
1300	3.3465629982421\\
1400	2.40959395938353\\
1500	1.73495704462981\\
1600	1.24920463673502\\
1700	0.899452945691367\\
1800	0.647624558637054\\
1900	0.466302957768905\\
2000	0.335747688261908\\
2100	0.241745217985689\\
2200	0.174061512445501\\
2300	0.125327857019322\\
2400	0.0902386261292168\\
2500	0.0649736605999189\\
2600	0.0467823675164157\\
2700	0.0336842635959427\\
2800	0.0242533602773057\\
2900	0.0174629165653377\\
3000	0.0125736578965223\\
3100	0.00905329143085903\\
3200	0.00651855541216332\\
3300	0.00469349351956207\\
3400	0.00337941154524303\\
3500	0.00243324558658121\\
3600	0.00175198670104299\\
3700	0.00126146633844066\\
3800	0.000908281622270051\\
3900	0.000653981386751308\\
4000	0.000470880004318751\\
};
\addlegendentry{expected rate}

\end{axis}
\end{tikzpicture}%
  \caption{Comparison of the randomized Kaczmarz method with and without mismatched adjoint in the underdetermined and consistent case but with true solution $\hat x\in\rg V^{T}$ and $\hat x\notin \rg A^{T}$. Plot shows the decay of the error.}
  \label{fig:underdetermined}
\end{figure}

For another illustration of the underdetermined case, we generated a toy problem from computerized tomography with the AIRtools package \cite{hansen2012air} as follows:
For a $50\times 50$ pixel image we generated a CT projection matrix for a parallel beam geometry with 36 equi-spaced angels and 150 rays per angle, leading to a projection matrix of size $5,400\times 2,500$ (the MATLAB command is \texttt{Afull = paralleltomo(50,0:5:180,150,70)}).
For the matrix $A$ for the forward projection we used every third row of the matrix while for $V$ (the backprojection) we used the averaged of three consecutive rows, thereby employing a simple model for detector bin width in the backprojection operation.
Then we eliminated the rows of $A$ and $V$ which correspond to zero-rows in $A$ which leaves us with two matrices of size $1,636\times 2,500$.
Then we generated a smooth image by
\begin{quote}
  \texttt{im = phantomgallery('ppower',N,0.3,1.3,155432);\\
    im = imfilter(im,fspecial('gaussian',16,4));\\ im =
    im/max(max(im));\\ x = im(:);}
\end{quote}
and generated the data by \texttt{b = A*x}.
We reconstructed $x$ by RKMA and RK (with the probabilities from Remark~\ref{rem:strohmer-vershynin}).
Figure~\ref{fig:ct_rek} shows the reconstructions after a quite small number of sweeps (one sweep corresponds to $m$ step of the methods, where $m$ is is the number of rows).
One sees that using a mismatched adjoint is beneficial in this setting: First, the iteration converges to a limit which is closer to the original image (which is due to the fact that this is closer to the range of $V^{T}$ that to that of $A^{T}$). Moreover, the initial iterates are better.
As expected, the reconstruction with $A^{T}$ and $A$ suffers from  Moire patterns. 
Using $V^{T}$ as adjoint avoids these artifact as the range of $V^{T}$ contains smoother functions, in some sense.
Finally, we note that applying RK using $V$ for both the forward and back-projection does also converge, but leads to an even worse reconstruction than using RK with $A$.

\begin{figure}[htb]
  \newcommand{\imwidth}{0.24\textwidth}
  \setlength\figurewidth{0.2\textwidth}
  \subfigure[Original]{\includegraphics[width=\imwidth]{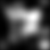}}
  \subfigure[Decay of error]{
%
%
\definecolor{mycolor1}{rgb}{0.00000,0.44700,0.74100}%
\definecolor{mycolor2}{rgb}{0.85000,0.32500,0.09800}%
\begin{tikzpicture}

\begin{axis}[%
width=0.951\figurewidth,
height=0.75\figurewidth,
at={(0\figurewidth,0\figurewidth)},
scale only axis,
xmin=0,
xmax=35000,
xlabel style={font=\color{white!15!black}},
xlabel={$k$},
ymode=log,
ymin=0.1,
ymax=100,
yminorticks=true,
ylabel style={font=\color{white!15!black}},
ylabel={$\|x_k-\hat x\|$},
axis background/.style={fill=white},
legend style={legend cell align=left, align=left, draw=white!15!black},
legend style={font=\scriptsize},
]
\addplot [color=mycolor1]
  table[row sep=crcr]{%
0	19.2426431948771\\
1636	6.91159148821129\\
3272	4.26940520332398\\
4908	3.49017829922446\\
6544	3.08127947970296\\
8180	2.82378092999374\\
9816	2.68742389244817\\
11452	2.58410192632922\\
13088	2.5000816046499\\
14724	2.44530021757056\\
16360	2.39684643470834\\
17996	2.36007102586635\\
19632	2.32804293391878\\
21268	2.30306850294023\\
22904	2.28097135682656\\
24540	2.26113816035802\\
26176	2.24698489751628\\
27812	2.2329243516769\\
29448	2.21832303574738\\
31084	2.20734056344153\\
32720	2.19449820904885\\
};
\addlegendentry{RK}

\addplot [color=mycolor2]
  table[row sep=crcr]{%
0	19.2426431948771\\
1636	5.19581559270871\\
3272	2.62780103173679\\
4908	1.6545411612526\\
6544	1.23644038139796\\
8180	1.04487069316001\\
9816	0.946712228256538\\
11452	0.885019764984337\\
13088	0.826132599762504\\
14724	0.780554801780955\\
16360	0.748933881399968\\
17996	0.717653176077198\\
19632	0.688101871557996\\
21268	0.666385573213715\\
22904	0.644609358224233\\
24540	0.625226271515553\\
26176	0.607742125283483\\
27812	0.594114023380855\\
29448	0.57798734124975\\
31084	0.568894012822039\\
32720	0.557593907514564\\
};
\addlegendentry{RKMA}

\end{axis}
\end{tikzpicture}%
}\\
  \subfigure[RK, 1 sweep]{\includegraphics[width=\imwidth]{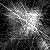}}
  \subfigure[RK, 3 sweeps]{\includegraphics[width=\imwidth]{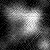}}
  \subfigure[RK, 10 sweeps]{\includegraphics[width=\imwidth]{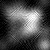}}
  \subfigure[RK, 20 sweeps]{\includegraphics[width=\imwidth]{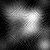}}\\
  \subfigure[RKMA, 1 sweep]{\includegraphics[width=\imwidth]{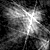}}
  \subfigure[RKMA, 3 sweeps]{\includegraphics[width=\imwidth]{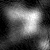}}
  \subfigure[RKMA, 10 sweeps]{\includegraphics[width=\imwidth]{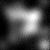}}
  \subfigure[RKMA, 20 sweeps]{\includegraphics[width=\imwidth]{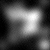}}\\
  \caption{Reconstruction for a toy CT example.}
  \label{fig:ct_rek}
\end{figure}
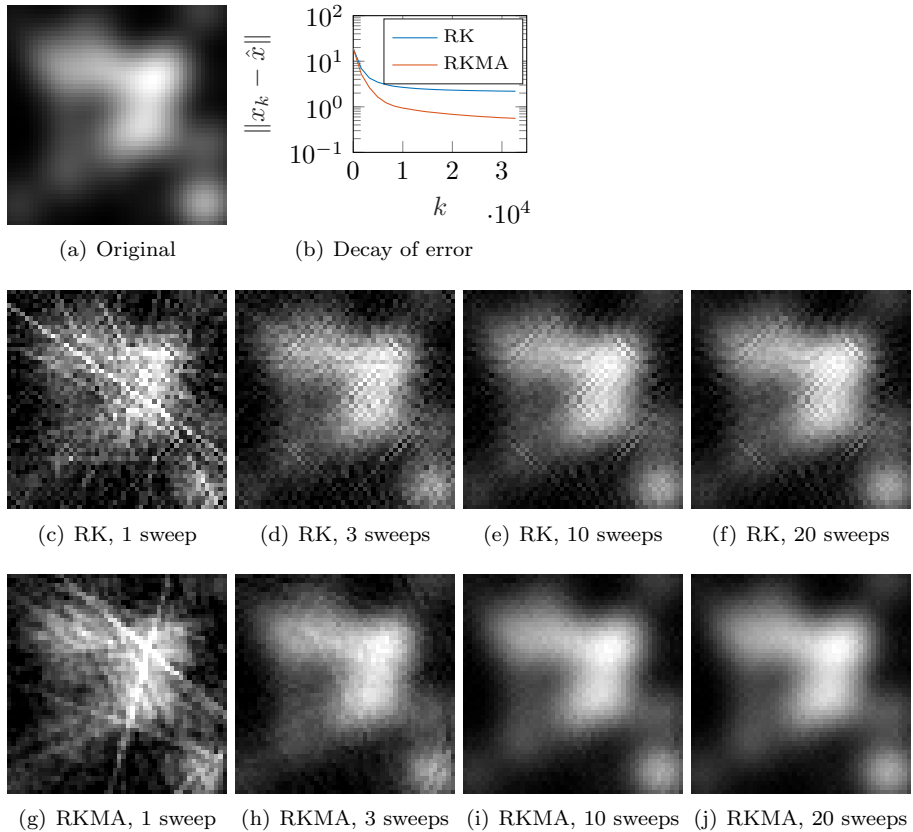

Finally, we illustrate that the optimization of the probabilities according to Section~\ref{sec:optim-probs} does indeed improve the practical performance.
We used $A,V \in\RR^{300\times 100}$ where $A$ is a random matix with Gaussian entries where the $i$th row has been scaled with the factor $2/(\sqrt{i}+2)$ and $V$ has been obtained from $A$ by setting 5\% randomly chosen entries of $A$ to zero.
We calculated optimized probabilities by the methods from Sections~\ref{sec:max_lambda_min} and~\ref{sec:min_norm}, respectively (initialized with uniform probabilities).
We applied RKMA with these optimized probabilities, uniform probabilities, and $p_{i}\propto \scp{a_{i}}{v_{i}}$. Figure~\ref{fig:optimize_lambda} shows that the optimized probabilites indeed outperform the uniform choice and the choice proportional to $\scp{a_{i}}{v_{i}}$. 
Table~\ref{tab:optimized-probabilities} shows the respective quantities for the different probabilities.
Although both approaches optimize different quantities and neither optimizes the asymptotic convergence rate, both probabilities are rather similar in practice, and, as shown in Figure~\ref{fig:optimize_lambda} on the right, the probabilities for the different optimization problems are quite similar.

\begin{table}[htb]
  \centering
  \begin{tabular}{lcccc}\toprule
                      & unif & row & $\max\lambda$ & $\min\norm{I-V^{T}DA}$\\\midrule
    $1-\lambda$        &0.998588 & 0.999079 & 0.997820 & 0.998311 \\
    $\rho(I-V^{T}DA)$  &0.997908 & 0.998352 & 0.997540 & 0.997327 \\
    $\norm{I-V^{T}DA}$ &0.998029 & 0.998485 & 0.997752 & 0.997439 \\\bottomrule
  \end{tabular}
  \caption{Quantities describing the convergence of RKMA for different probabilities.}
  \label{tab:optimized-probabilities}
\end{table}

\begin{figure}[htb]
  \centering
  \setlength\figurewidth{0.35\textwidth}
%
%
\definecolor{mycolor1}{rgb}{0.00000,0.44700,0.74100}%
\definecolor{mycolor2}{rgb}{0.85000,0.32500,0.09800}%
\definecolor{mycolor3}{rgb}{0.92900,0.69400,0.12500}%
\definecolor{mycolor4}{rgb}{0.49400,0.18400,0.55600}%
\begin{tikzpicture}

\begin{axis}[%
width=0.951\figurewidth,
height=1.048\figurewidth,
at={(0\figurewidth,0\figurewidth)},
scale only axis,
xmin=0,
xmax=8000,
xlabel style={font=\color{white!15!black}},
xlabel={$k$},
ymode=log,
ymin=1e-09,
ymax=10,
yminorticks=true,
ylabel style={font=\color{white!15!black}},
ylabel={$\|x_k-\hat x\|$},
axis background/.style={fill=white},
legend style={legend cell align=left, align=left, draw=white!15!black},
legend style={font=\scriptsize},
]
\addplot [color=mycolor1]
  table[row sep=crcr]{%
0	9.18776673343378\\
300	2.8998894821406\\
600	1.2526289642923\\
900	0.514427203660014\\
1200	0.258059250283638\\
1500	0.125495358598903\\
1800	0.0558572732098953\\
2100	0.0250409874824853\\
2400	0.0102519666861141\\
2700	0.0041665231410138\\
3000	0.00201627588953081\\
3300	0.0010336370104877\\
3600	0.000445848273651003\\
3900	0.000235482707703009\\
4200	0.000103425262911142\\
4500	4.72686259628947e-05\\
4800	1.9244895142441e-05\\
5100	7.51768146150301e-06\\
5400	3.26148014558699e-06\\
5700	1.49526316556784e-06\\
6000	6.32783202703965e-07\\
6300	2.82482701993938e-07\\
6600	1.55903785395132e-07\\
6900	6.94841384217729e-08\\
7200	3.40498214493985e-08\\
7500	1.69002256290857e-08\\
7800	7.85433152876205e-09\\
};
\addlegendentry{opt $\|I-V^TDA\|$}

\addplot [color=mycolor2]
  table[row sep=crcr]{%
0	9.18776673343378\\
300	2.77104945494776\\
600	1.13127345229608\\
900	0.492413764651897\\
1200	0.226464407599363\\
1500	0.11686742370087\\
1800	0.0379309007694225\\
2100	0.0183071566028428\\
2400	0.0085459775132953\\
2700	0.00433589168986686\\
3000	0.00200272029875946\\
3300	0.000929962272060166\\
3600	0.000455067557243009\\
3900	0.000183159775978073\\
4200	8.50521955857424e-05\\
4500	4.1578237420205e-05\\
4800	1.47667553183786e-05\\
5100	6.56984891909457e-06\\
5400	2.75713649706943e-06\\
5700	1.4312448449168e-06\\
6000	7.28518418867223e-07\\
6300	3.57626733345161e-07\\
6600	1.63859610665837e-07\\
6900	7.56262775014824e-08\\
7200	3.93808385727371e-08\\
7500	2.05029692345983e-08\\
7800	1.03516188079607e-08\\
};
\addlegendentry{opt $\lambda$}

\addplot [color=mycolor3]
  table[row sep=crcr]{%
0	9.18776673343378\\
300	2.72451652979289\\
600	1.0810585852168\\
900	0.423275694357549\\
1200	0.169603913146267\\
1500	0.0774757838541698\\
1800	0.034742060476637\\
2100	0.0162765849878355\\
2400	0.00906319614617254\\
2700	0.00492791229195241\\
3000	0.0027680096525542\\
3300	0.00152589522615802\\
3600	0.000826595594376192\\
3900	0.00045242682211757\\
4200	0.000215259316378454\\
4500	0.000130431400878885\\
4800	7.17522921111452e-05\\
5100	3.95188052838762e-05\\
5400	2.18430964390456e-05\\
5700	1.15208761914229e-05\\
6000	6.56304654689568e-06\\
6300	4.17119056585959e-06\\
6600	2.21432510406495e-06\\
6900	1.15820162845933e-06\\
7200	6.46969047775068e-07\\
7500	3.85473623348434e-07\\
7800	1.95150159417086e-07\\
};
\addlegendentry{unif}

\addplot [color=mycolor4]
  table[row sep=crcr]{%
0	9.18776673343378\\
300	2.49588032722787\\
600	1.41760569994378\\
900	0.732682657671624\\
1200	0.457731070062441\\
1500	0.282760626589907\\
1800	0.180383016855586\\
2100	0.105158412454391\\
2400	0.0597000358923658\\
2700	0.0360280684683189\\
3000	0.021664693320148\\
3300	0.0142115518411541\\
3600	0.0088183207495205\\
3900	0.00553242011799617\\
4200	0.00343133157928325\\
4500	0.00226062100548556\\
4800	0.00113259074228682\\
5100	0.000666299486613425\\
5400	0.000426142672075336\\
5700	0.00028581685396183\\
6000	0.000158491202816528\\
6300	0.000101789213646755\\
6600	6.12689855638568e-05\\
6900	3.89249459780143e-05\\
7200	2.25030813309067e-05\\
7500	1.33546155452884e-05\\
7800	8.55296719687011e-06\\
};
\addlegendentry{$\propto \langle a_i,v_i\rangle$}

\end{axis}
\end{tikzpicture}%
  \input{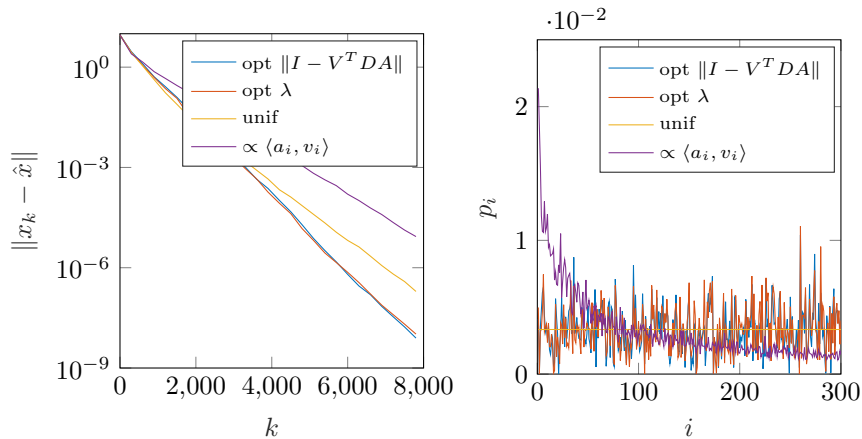}
  \caption{A sample run of RKMA on a consistent system with different
    probabilities: ``unif'' refers to the uniform probabilities
    $p_{i}= 1/m$, ``$\propto\scp{a_{i}}{v_{i}}$'' uses
    $p_{i} = \scp{a_{i}}{v_{i}}/\norm[V]{A}^{2}$
    (cf. Remark~\ref{rem:strohmer-vershynin}) and ``opt $\norm{I-V^{T}DA}$'' and ``opt $\lambda$'' refer to
    probabilities obtained by the methods from
    Section~\ref{sec:optim-probs}.}
  \label{fig:optimize_lambda}
\end{figure}


\section{Conclusion}
\label{sec:conclusion}

We derived several results on the convergence of the randomized Kaczmarz method with mismatched adjoint and could show that the method converges linearly when the mismatch is not too large.
The results are a little bit more complicated compared to the case of no mismatch due to the asymmetry of the matrix $I-V^{T}DA$. In particular, estimates for the norm of the expected error and the expectation of the norm of the error are different in this case.
We were also able to characterize the asymptotic convergence rate of RKMA and numerical experiments indicate that this estimate of the rate is indeed quite sharp.
In the underdetermined case one may even take advantage of the use of a mismatched adjoint to drive the randomized Kaczmarz method to a solution in the subspace $\rg V^{T}$.
This last point may be important for algebraic reconstruction techniques in computerized tomography where mismatched projector pairs are often employed.
Using the conditions derived here, a thorough study of commonly used mismatched projector pairs could be performed to determine what pairs have guaranteed asymptotic convergence properties.

\section*{Acknowledgements}

The authors thank Emil Sidky for valuable discussions. This material was based upon work partially supported by the National Science Foundation under Grant DMS-1127914 to the Statistical and Applied Mathematical Sciences Institute.  Any opinions, findings, and conclusions or recommendations expressed in this material
are  those  of  the  author(s)  and  do  not  necessarily  reflect  the  views  of  the  National  Science
Foundation.

\bibliographystyle{plain}
\bibliography{./literature}

\begin{thebibliography}{10}

\bibitem{agaskar2014randomized}
Ameya Agaskar, Chuang Wang, and Yue~M Lu.
\newblock Randomized {K}aczmarz algorithms: {E}xact {MSE} analysis and optimal
  sampling probabilities.
\newblock In {\em Signal and Information Processing (GlobalSIP), 2014 IEEE
  Global Conference on}, pages 389--393. IEEE, 2014.

\bibitem{de2004distance}
Bruno De~Man and Samit Basu.
\newblock Distance-driven projection and backprojection in three dimensions.
\newblock {\em Physics in Medicine \& Biology}, 49(11):2463, 2004.

\bibitem{golub2013matrix}
Gene~H. Golub and Charles~F. Van~Loan.
\newblock {\em Matrix computations}.
\newblock Johns Hopkins Studies in the Mathematical Sciences. Johns Hopkins
  University Press, Baltimore, MD, fourth edition, 2013.

\bibitem{gordon1970algebraic}
Richard Gordon, Robert Bender, and Gabor~T Herman.
\newblock Algebraic reconstruction techniques ({ART}) for three-dimensional
  electron microscopy and {X}-ray photography.
\newblock {\em Journal of theoretical Biology}, 29(3):471--481, 1970.

\bibitem{hansen2012air}
Per~Christian Hansen and Maria Saxild-Hansen.
\newblock {AIR} tools--a {MATLAB} package of algebraic iterative reconstruction
  methods.
\newblock {\em Journal of Computational and Applied Mathematics},
  236(8):2167--2178, 2012.

\bibitem{kak2002principles}
Avinash~C Kak, Malcolm Slaney, and Ge~Wang.
\newblock Principles of computerized tomographic imaging.
\newblock {\em Medical Physics}, 29(1):107--107, 2002.

\bibitem{liu2014asynchronous}
Ji~Liu, Stephen~J Wright, and Srikrishna Sridhar.
\newblock An asynchronous parallel randomized {K}aczmarz algorithm.
\newblock {\em arXiv preprint arXiv:1401.4780}, 2014.

\bibitem{needell2010noisy}
Deanna Needell.
\newblock Randomized {K}aczmarz solver for noisy linear systems.
\newblock {\em BIT Numerical Mathematics}, 50(2):395--403, Jun 2010.

\bibitem{siddon1985fast}
Robert~L Siddon.
\newblock Fast calculation of the exact radiological path for a
  three-dimensional {CT} array.
\newblock {\em Medical physics}, 12(2):252--255, 1985.

\bibitem{strohmer2009randomized}
Thomas Strohmer and Roman Vershynin.
\newblock A randomized {K}aczmarz algorithm with exponential convergence.
\newblock {\em Journal of Fourier Analysis and Applications}, 15(2):262--278,
  2009.

\bibitem{watson1992characterization}
G~Alistair Watson.
\newblock Characterization of the subdifferential of some matrix norms.
\newblock {\em Linear algebra and its applications}, 170:33--45, 1992.

\bibitem{zeng2000unmatched}
Gengsheng~L Zeng and Grant~T Gullberg.
\newblock Unmatched projector/backprojector pairs in an iterative
  reconstruction algorithm.
\newblock {\em IEEE transactions on medical imaging}, 19(5):548--555, 2000.

\bibitem{zouzias2013randomized}
Anastasios Zouzias and Nikolaos~M Freris.
\newblock Randomized extended {K}aczmarz for solving least squares.
\newblock {\em SIAM Journal on Matrix Analysis and Applications},
  34(2):773--793, 2013.

\end{thebibliography}

\end{document}